\documentclass[]{siamltex}
\usepackage{amssymb,amsmath}
\usepackage{color}
\usepackage{graphics}
\usepackage{graphicx}

\newtheorem{remark}[theorem]{Remark}

\newcommand{\be}{\begin{equation}}
\newcommand{\ee}{\end{equation}}

\title{\bf Decay Rates and Eigenvalue Asymptotics  for Abstract Strongly Coupled Hyperbolic Equations with Infinite Memory \thanks{This work was supported by the Natural Science Foundation of
	China grant NSFC-62073236 and NSFC-61773277.}}

\author{
	Hai-E Zhang\thanks{Department of Basic Science, Tangshan University, Tangshan 063000, China;~School of Mathematics,  Tianjin University, Tianjin 300354, China ({haiezhang@126.com; haiezhang@tju.edu.cn}).}
\and Gen-Qi Xu\thanks{School of Mathematics,  Tianjin University, Tianjin 300354, China ({gqxu@tju.edu.cn}).}
\and Zhong-Jie Han\thanks{Corresponding Author. School of Mathematics, Tianjin University, Tianjin 300354, China   ({zjhan@tju.edu.cn}). }
	}

\begin{document}
		
		\maketitle
	
\begin{abstract}
	In this paper, the asymptotic behavior of  abstract strongly coupled  hyperbolic equations with one infinite memory term is investigated, one specific case of which is the model for  describing the dynamical behaviour of {magnetic effected} piezoelectric beams. A fractional operator is involved in  the memory term depending on the parameter $a\in [0,1)$.  By means of frequency domain analysis,  it is proved that the system can be indirectly stabilized polynomially by only one infinite memory term located on one of these two strongly coupled PDEs, and the explicit decay rates {given as $t^{-\frac{1}{2-2a}}$ are only dependent on the parameter $a$.} When considering  the exponentially decreasing kernel, a detailed spectral analysis for the system operator is further provided. Specifically, the  asymptotic expressions of the eigenvalues of the system operator are derived. Based on the expressions of the eigenvalues,  the optimality of the obtained decay rates is further verified for this system.	 \end{abstract}
	
	\begin{keywords}
Polynomial stability, strongly coupled system, fractional operator, infinite memory, frequency domain method, spectral analysis.
	\end{keywords}
	
	
	\pagestyle{myheadings}
	\thispagestyle{plain}
	\markboth{H. E ZHANG, G. Q. XU AND Z. J. HAN}{Decay rates and eigenvalue asymptotics  for abstract strongly coupled equations with memory}

\section{Introduction}

\noindent

  Let $(H,\langle\cdot,\cdot\rangle,\|\cdot\|)$ be a complex Hilbert space. {Suppose that $A: \mathcal{D}(A)\to H$ is} a  strictly positive and self-adjoint operator  with compact resolvent on $H$. Consider the following abstract system composing of two strongly coupled equations with one infinite memory term, whose dynamic behavior is described by
\begin{equation}\label{alpha-sys1}
\left\{
\begin{array}{ll}
\rho v_{tt}(t)=-\alpha A v(t)+\gamma\beta A p(t)+\int_{0}^{\infty}g(s)A^a v(t-s)ds,&~t>0,\\
\mu p_{tt}(t)=-\beta A p(t)+\gamma\beta A v(t),&~t>0,\\
v(0)=v_0,\; v_t(0)=v_1, \; p(0)=p_0, \; p_{t}(0)=p_{1},\\
v(-s)=h(s),~&~s>0,
\end{array}
\right.
\end{equation}
where $a\in [0,1)$ leads to a fractional operator $A^a$, $\rho,\mu,\alpha,\beta,\gamma>0$ satisfying $\alpha> \gamma^2\beta$ and set $\alpha_1:=\alpha-\gamma^2\beta>0$. $h(s)$ is the memory history and $g(s)$ is the memory kernel function. Assume that $g(s)$ satisfies the following hypotheses:

(A1)  $g\in L^1(\mathbb{R}_+)\cap H^1(\mathbb{R_+})$ satisfies {$0<\zeta:=\int_{0}^{\infty}g(s)ds<\infty$} and $g(s)>0$ for $s\in \mathbb{R}_+$;

(A2) For any $s\in \mathbb{R}_+$, $g'(s)<0$ and there exist two constants $k_0>0$ and $k_1>0$ such that {$-k_0g(s)\leq g'(s)\leq-k_1g(s)$.}

Note that system \eqref{alpha-sys1} is derived from the {fully magnetic effected} piezoelectric beam model proposed by Morris and \"{O}zer in \cite{ma1}. The concrete model  can be deduced by setting $A:=-\frac{\partial^2}{\partial x^2}$ in \eqref{alpha-sys1}.
Especially,  this model has found its numerous applications in industrial fields such as ultrasonic welders, micro-sensors and inchworm robots. As a result, the stability issue  of the piezoelectric beam system has  attracted the attention of many scholars.  The polynomial stability results have been reported in \cite{ma2, ma3,zhangsiap}, while  the exponential stability results  can be found in \cite{santos2021,ma1,am4,zhangxuhan1} and the references therein.

Memory terms  usually appear as Boltzmann integrals in mathematical models, which  serve as   passive controls in   engineering practical problems \cite{tzou}. Two types of Boltzmann integrals exist: infinite memory (history memory function is not zero) and finite memory (history memory function is always zero). Guo et al. in \cite{guo-zhang-2012} investigated the spectral distribution of one-dimensional wave equation with two kinds of Boltzmann damping. Further details on  Boltzmann integrals can be found in \cite{wang-guo-b}.  Rivera et al. in \cite{rivera2010}  studied the polynomial decay rate of a class of abstract second-order systems with a fractional operator in the infinite memory. {Xu \cite{xu2022} developed the resolvent family theory to investigate the linear evolution process with infinite memory.}
If  the fractional order $a$ in the memory term $\int_{0}^{\infty}g(s)A^a v(t-s)ds$ is chosen as $a=1$, it is referred to as `` viscoelastic infinite memory". Many authors have addressed this  problem \cite{cava,f-p-2002,wang-guo-2009,liu-liu-1996,liu-liu-2002,r-o-2000,wang-wang-2014,yang}. For instance, \cite{liu-liu-1996} studied a linear abstract system when the memory kernel decayed exponentially and discussed the spectrum of the associated semigroup generator. \cite{wang-wang-2014} analysed the exponential stability of a wave equation with viscoelastic damping. \cite{yang} explored the existence and long-time behavior of solutions to  an Euler-Bernoulli viscoelastic equation with a delay term in the internal feedbacks.

In recent years, infinite memory terms have been increasingly incorporated into coupled hyperbolic systems as a form of ``indirect damping" to stabilize systems. Timoshenko beams, in particular, have attracted significant attention as a class of weakly coupled hyperbolic systems with memory-related issues. Astudillo et al. \cite{a-o-2021} investigated the asymptotic behavior of the Timoshenko beam system with infinite memory. {Meanwhile, Rivera et al. \cite{jemr}  investigated the stability  for a Timoshenko system with historical memory. They employed the frequency domain method (see  \cite{gearhart,huang,pruss,liu-zheng}) and the Lyapunov functional method to  show the exponential or polynomial decay of the system with  same or different wave speeds, respectively.} Giorgi et al. \cite{gi}, on the other hand, employed energy estimates to prove exponential decay for a semilinear  Mindlin-Timoshenko beam with memory. Additionally, there are numerous studies on coupled systems with finite memories. For example, Wang et al. \cite{wangjunmin} discussed the exponential stability of a  hyperbolic system derived from thermalelastic equation using the Riesz basis method, while Jin et al. \cite{2014coupled} explored the optimal decay rate of energy for an abstract second-order coupled system via the Lyapunov function method. For more information, please see \cite{SA3,2022chentouf,2018feng,rao,rivera2010,wangjunmin}.

However, to the best of our knowledge, there are no results available regarding the spectrum and stability of strongly coupled hyperbolic systems with infinite memory that involve fractional operators. Generally, as the coupling term becomes stronger, the spectral analysis becomes more complex. The calculation for the spectrum always requires careful and meticulous analysis due to the complexity of the system, particularly when infinite memory is involved in such coupled systems. All these factors contribute to the difficulty of the stability analysis.
This study focuses on an abstract strongly coupled hyperbolic system, denoted by \eqref{alpha-sys1}, where damping occurs through a single memory term in the principal operator. The memory term involves a fractional order $a\in [0,1)$.   The main objectives of this research are to analyse the spectral properties and stability of system (\ref{alpha-sys1}). Specifically, we aim to address the following issues:

Q1. Can this type of strongly coupled system (\ref{alpha-sys1}) be indirectly stabilized via the single memory term? Moreover, what is the relationship between the decay rate of the solution and the parameter $a$ in the memory term?

Q2. What is the asymptotic distribution of the spectrum of the system operator associated with system (\ref{alpha-sys1})? In particular, what explicit relationship exists between the real and imaginary parts of the spectrum?

 The main contributions  are listed as follows:

1.  The decay rates for solutions to system (\ref{alpha-sys1}) are primarily addressed using the frequency domain method alongside multiplier techniques. By providing a refined estimate for the norm of the corresponding resolvent operator along the imaginary axis, we obtain explicit polynomial decay rates  $t^{-\frac{1}{2-2a}}$  for the solutions to system (\ref{alpha-sys1}) under smooth initial states. Especially, these decay rates are observed to be dependent on the fractional order
$a$ in the infinite memory term.

2.  We conduct a  comprehensive spectral analysis on the system operator for system (\ref{alpha-sys1}), in which the memory term's kernel is exponential type. The asymptotic expressions of the eigenvalues of the system operator are presented.
Moreover, the optimality of the obtained decay rates $t^{-\frac{1}{2-2a}}$ is  further confirmed through the relationship between the real and imaginary parts of the spectrum identified from the expressions of the eigenvalues.

{The remainder of this paper is organized as follows.  Section 2 presents the  space settings and proves the well-posedness of the system by semigroup theories. Section 3 discusses the polynomial decay rates of the system by the frequency domain method.
In section 4, we provide a detailed asymptotic spectral analysis, which serves as the basis for verifying the optimality of the obtained polynomial decay rates.  Lastly, a concluding remark is presented in section 5.}

\section{Well-posedness}
\noindent
This section is devoted to investigating  the well-posedness of system (\ref{alpha-sys1}). To do this,  let us first reformulate  system \eqref{alpha-sys1} into an
abstract evolution equation in an appropriate Hilbert space setting. {Thanks to the assumption of the operator $A$, we easily know that}
the spectrum of $A$ only contains isolated eigenvalues $\{\xi_k, k\in \mathbb{N}\}$ and a sequence of eigenfunctions $\{\Psi_k, k\in \mathbb{N}\}\subset H$ such that
\begin{equation}\label{egienA}
0<\xi_{1}<\xi_{2}<\cdots<\xi_k<\xi_{k+1}<\cdots, \quad \lim_{k\to\infty}\xi_k=+\infty
\end{equation}
and $A\Psi_k=\xi_k \Psi_k$ for all $k\in \mathbb{N}$. Moreover,
the function sequence  $\{\Psi_k, k\in \mathbb{N}\}$ forms an orthonormal basis for $\mathbb{X}$.

%
We define the fractional operator and the fractional space as follows. For $\vartheta\in [0,1]$,
$$A^{\vartheta}f=\sum\limits_{k=1}^\infty\xi_k^{\vartheta}\langle f,\Psi_k\rangle\Psi_k$$ and
$$D(A^{\vartheta})=\big\{f\in H~\big|~\sum\limits_{k=1}^{\infty}\xi_k^{2\vartheta}|\langle f,\Psi_k\rangle|^2<\infty\big\}.$$

For $r\in \mathbb{R}$, we consider the scale of Hilbert spaces $\mathcal{D}(A^{\frac{r}{2}})$ endowed with the usual inner product $\langle v_1, v_2\rangle_{\mathcal{D}(A^{\frac{r}{2}})} =\langle A^{\frac{r}{2}}v_1, A^{\frac{r}{2}}v_2\rangle$. The embedding $\mathcal{D}(A^{\frac{r_1}{2}}) \hookrightarrow \mathcal{D}(A^{\frac{r_2}{2}})$ is compact whenever $r_1> r_2$. Then  for $a\in [0,1)$, there is a constant $\kappa\in(0,\alpha_1)$ such that
\begin{eqnarray}\label{norm-equi}
\alpha_1\|A^{\frac{1}{2}}u\|^2-\zeta {\|A^{\frac{a}{2}}u\|^2}>\kappa \|A^{\frac{1}{2}}u\|^2,~\forall~u\in \mathcal{D}(A^{\frac{1}{2}}),
\end{eqnarray}
where $\zeta$ is given as in the assumption (A1).
This  implies that
\begin{eqnarray}\label{norm-equi-1}
\|A^{\frac{a}{2}}u\|^2<\frac{\alpha_1}{\zeta}\|A^{\frac{1}{2}}u\|^2
\end{eqnarray}
and induces the equivalent norms
\begin{eqnarray}\label{norm-equi-2}
\alpha_1\|A^{\frac{1}{2}}u\|^2-\zeta \|A^{\frac{a}{2}}u\|^2\sim\|A^{\frac{1}{2}}u\|^2.
\end{eqnarray}
{By the approach of Dafermos \cite{daf,dafermos}, let us introduce a new variable  $\eta(t,s)$ for system (\ref{alpha-sys1})
\begin{equation}\nonumber
\eta(t,s)=v(t)-v(t-s),~t,s>0.
\end{equation}

It is easy to verify that $\eta_t(t,s)=v_t(t)-\eta_s(t,s)$ and $\eta(t,0)=0, \; t>0$.
Thus, system (\ref{alpha-sys1}) can be  rewritten as follows:
\begin{eqnarray}\label{alpha-sys2}
\left\{
\begin{array}{ll}
\rho v_{tt}(t)=-\alpha A v(t)+\gamma\beta A p(t)+\zeta A^a v(t)-\int_{0}^{\infty}g(s)A^a \eta(t,s)ds,&~t>0,\\
\mu p_{tt}(t)=-\beta A p(t)+\gamma\beta A v(t),&~t>0,\\
\eta_t(t,s)=v_t(t)-\eta_s(t,s),&~t,s>0,\\
v(0)=v_0,~v_t(0)=v_1,~p(0)=p_0,~p_{t}(0)=p_{1},\\
v(-s)=h(s),&~s>0.\\
\end{array}
\right.
\end{eqnarray}

We define the weighted space $L_{g}^2(\mathbb{R}_+,D(A^{\frac{a}{2}}))$ by
\[
L_{g}^2(\mathbb{R}_+,D(A^{\frac{a}{2}})):=\left\{w(s)\biggm|\begin{array}{c}
                                                                \text{for any}~s\in \mathbb{R}_+,w(s)\in D(A^{\frac{a}{2}}), \\
                                                                \int_{0}^{\infty}g(s)\|A^{\frac{a}{2}}w(s)\|^2 ds<\infty
                                                                \end{array}\right\}
\]
equipped with the inner product
\[
\langle w,\tilde{w}\rangle_{L_{g,a}^2}=\int_{0}^{\infty}g(s)\langle A^{\frac{a}{2}} w(s),A^{\frac{a}{2}}\tilde{w}(s)\rangle ds.
\]
Set the state space as \begin{equation}\nonumber
\mathcal{H}:=(D(A^{\frac{1}{2}})\times H)^2 \times L_{g}^2(\mathbb{R}_+,D(A^{\frac{a}{2}}))\end{equation}
endowed with the inner product
\begin{eqnarray}\label{inner-product}
&&\left\langle(v,u,p,q,\eta)^T,(\tilde{v},\tilde{u},\tilde{p},\tilde{q},\tilde{\eta})^T\right\rangle_{\mathcal{H}}=\alpha_1\langle A^{\frac{1}{2}} v, A^{\frac{1}{2}} {\tilde{v}}\rangle-\zeta \langle A^{\frac{a}{2}} v, A^{\frac{a}{2}} {\tilde{v}}\rangle+\rho\langle u, {\tilde{u}}\rangle\cr
&&+\beta\langle\gamma A^{\frac{1}{2}}  v-A^{\frac{1}{2}}  p,\gamma A^{\frac{1}{2}} {\tilde{v}}-A^{\frac{1}{2}} {\tilde{p}}\rangle+\mu \langle q, {\tilde{q}}\rangle+\int_{0}^{\infty}g(s)\langle A^{\frac{a}{2}} \eta(s), A^{\frac{a}{2}}{\tilde{\eta}}(s)\rangle ds,
\end{eqnarray}
which induces the norm $\|\cdot\|_{\mathcal{H}}$.
By \eqref{norm-equi} and the assumption (A1),  it is easy to check that $(\mathcal{H}, \langle\cdot,\cdot\rangle_{\mathcal{H}})$ is a Hilbert space.

Define the system operator $\mathcal{A}$ in $\mathcal{H}$ by
\begin{equation}\label{alpha-A}
\mathcal{A}\left(
   \begin{array}{c}
     v \\
     u\\
     p \\
     q\\
     \eta
   \end{array}
 \right)=\left(
           \begin{array}{c}
             u \\
             \frac{1}{\rho}(-\alpha A v+\gamma\beta A p+\zeta A^a v-\int_{0}^{\infty}g(s)A^a \eta(s)ds)\\
             q \\
             \frac{1}{\mu}(-\beta A p+\gamma\beta A v)\\
             u-\eta_s(s)

           \end{array}
         \right)
\end{equation}
with domain
\begin{equation}\nonumber
\mathcal{D}(\mathcal{A})=\big(D(A)\times D(A^\frac{1}{2})\big)^2\times L_{g}^2(\mathbb{R}_+,D(A^\frac{a}{2})).
\end{equation}

 Then, the problem (\ref{alpha-sys2}) can be reformulated into the following abstract evolution equation
\begin{equation}\label{eq2.8}
\left\{
\begin{array}{ll}
\frac{dX(t)}{dt}=\mathcal{A}X(t),&t>0,\\
X(0)=X_0,
\end{array}
\right.
\end{equation}
where $
X(t)=\left(
v(t),
v_t(t),
p(t),
p_t(t),
\eta(s,t)\right)^{T},\; X_0=\left(
v_0,
v_1,
p_0,
p_1,
\eta_0
\right)^{T}$.

 We have the following result on the well-posedness of system (\ref{eq2.8}).
\begin{theorem}
Assume that the assumptions (A1)-(A2) hold. Then 
%
%
the operator $\mathcal{A}$ generates a $C_0$ semigroup of contractions $\{ e^{\mathcal{A}t}\}_{t\geq 0}$ on Hilbert space $\mathcal{H}$.
\end{theorem}

\begin{proof}
For any $W=(v,u,p,q,\eta)^T\in \mathcal{D}(\mathcal{A})$, integration by parts yields that
\begin{eqnarray}\label{alpha-E}
\Re\langle\mathcal{A}W,W\rangle_{\mathcal{H}}&=&\Re\bigg(\alpha_1\langle A^{\frac{1}{2}} u , A^{\frac{1}{2}} v\rangle-\zeta \langle A^{\frac{a}{2}} u, A^{\frac{a}{2}} v\rangle\cr
&&+\langle-\alpha A v+\gamma\beta A p+\zeta A^a v-\int_{0}^{\infty}g(s)A^a \eta(s)ds,u\rangle\cr
&&+\beta\langle\gamma A^{\frac{1}{2}} u-A^{\frac{1}{2}} q,\gamma A^{\frac{1}{2}} v-A^{\frac{1}{2}} p\rangle\cr
&&+\langle-\beta A p+\gamma\beta A v,q\rangle\cr
&&+\langle\int_{0}^{\infty}g(s) A^{\frac{a}{2}}(u-\eta_s(s))ds, A^{\frac{a}{2}}\eta(s)\rangle\bigg)\cr
&=&\frac{1}{2}\int_{0}^{\infty}g'(s)\|A^{\frac{a}{2}} \eta(s)\|^2ds\leq 0,
\end{eqnarray}
which implies that $\mathcal{A}$ is dissipative.
Moreover, it remains to show that 0 belongs to the resolvent
set $\varrho(\mathcal{A})$ and then the result follows according to the well-known Lumer-Phillips theorem \cite{brez,pazy}. Indeed,
for any $F=(f_1,f_2,z_1,z_2,\nu)^T\in \mathcal{H}$, let us discuss the solvability of the equation
\begin{align}
 \label{alpha-resolvent}\mathcal{A}W=F,~~\text{for}~W\in \mathcal{D}(\mathcal{A}),
\end{align}
that is,
\begin{align}
 \label{alpha-w-p1}u= f_1,\\
 \label{alpha-w-p2}             \frac{1}{\rho}[-\alpha A v+\gamma\beta A p+\zeta A^a v-\int_{0}^{\infty}g(s)A^a \eta(s)ds]=f_2,\\
\label{alpha-w-p3}            q=z_1, \\
\label{alpha-w-p4}            \frac{1}{\mu}[-\beta A p+\gamma\beta A v]=z_2,\\
\label{alpha-w-p5}            u-\eta_s=\nu.
\end{align}
In what follows, it is sufficient to show that \eqref{alpha-resolvent} has a solution belongs to $D(\mathcal{A})$ and its norm can be controlled by $\|F\|_{\mathcal{H}}$.
Solving the ODE $\eqref{alpha-w-p5}$, along with $\eta(t,0)=0$ and $\eqref{alpha-w-p1}$, it follows that
\begin{align}
\label{alpha-w-solu}\eta=\int_{0}^{s}[f_1-\nu(r)]dr.
\end{align}

Taking account of $\alpha_1=\alpha-\gamma^2\beta$, it is obvious that the equations $\eqref{alpha-w-p1}-\eqref{alpha-w-p5}$ can be transformed into the following ones
\begin{align}
\label{alpha-w-p2-1} &\beta(\gamma A v-A p)=\mu z_2,\\
\label{alpha-w-p4-1} & \alpha_1 A v-\zeta A^a v+\gamma\beta(\gamma A v-A p)=-\rho[f_2+\int_{0}^{\infty}g(s)A^a \eta(s)ds],
\end{align}
where $h_1=\mu z_2,h_2=-\rho[f_2+\int_{0}^{\infty}g(s)A^a \eta(s)ds]$.

Let $(\varphi,\psi)\in (D(A^{\frac{1}{2}}))^2$. Multiplying $\eqref{alpha-w-p2-1}$ and $\eqref{alpha-w-p4-1}$ by $\psi$ and $\varphi$ respectively and taking  inner product yields that
\begin{align}
\nonumber &\alpha_1\langle A^{\frac{1}{2}}v,A^{\frac{1}{2}}{\varphi}\rangle-\zeta\langle A^{\frac{a}{2}}v,A^{\frac{a}{2}}{\varphi}\rangle+\beta\langle \gamma A^{\frac{1}{2}}v-A^{\frac{1}{2}}p,\gamma A^{\frac{1}{2}}{\varphi}-A^{\frac{1}{2}}{\psi}\rangle\\
\label{alpha-w-p2-2}=&\langle h_2, {\varphi}\rangle-\langle h_1, {\psi}\rangle.
\end{align}
We introduce the space $V:=(D(A^{\frac{1}{2}}))^2$ endowed with the norm
\[
\|(f,\widetilde{f})\|_{V}^2=\alpha_1\|A^{\frac{1}{2}}f\|^2-\zeta\|A^{\frac{a}{2}}f\|^2+\beta\|\gamma A^{\frac{1}{2}}f-A^{\frac{1}{2}}\widetilde{f}\|^2.
\]
Define the bilinear functional $B$ by
\begin{align}
\nonumber
 &B((v,p),(\varphi,\psi)):=\alpha_1\langle A^{\frac{1}{2}}v,A^{\frac{1}{2}}{\varphi}\rangle-\zeta\langle A^{\frac{a}{2}}v,A^{\frac{a}{2}}{\varphi}\rangle+\beta\langle \gamma A^{\frac{1}{2}}v-A^{\frac{1}{2}}p,\gamma A^{\frac{1}{2}}{\varphi}-A^{\frac{1}{2}}{\psi}\rangle.
\end{align}
It is obvious that $B$ is bounded and coercive in $V\times V$.
Introduce the functional $G$ defined on $V$ as
\begin{align}
 \nonumber G(\varphi,\psi)&=\langle h_2,{\varphi}\rangle-\langle h_1, {\psi}\rangle.
\end{align}
Using H\"{o}lder inequality and Poincar\'{e} inequality, simple calculation shows that $F$ is bounded in $V$.
Then $\eqref{alpha-w-p2-2}$ can be placed in the variational equation as follows.
\begin{align}
\label{alpha-variation} &B((v,p),(\varphi,\psi))=G(\varphi,\psi).
\end{align}
 Then it follows by Lax-Millgram theorem that there is a unique solution $(v,p)\in V$ of $\eqref{alpha-variation}$. Therefore, the equations \eqref{alpha-w-p2} and \eqref{alpha-w-p4} hold in a weak sense and hence $\eqref{alpha-resolvent}$ admits a unique solution $W\in \mathcal{D}(\mathcal{A})$. In addition, replacing $({\varphi}, {\psi})$ by $(v,p)$ in \eqref{alpha-w-p2-2}, and applying H\"{o}lder inequality and Young inequality, we have that there exist two positive constants $\varepsilon$ and $k_\varepsilon$ such that
$$
\alpha_1\|A^{\frac{1}{2}}v\|^2-\zeta\|A^{\frac{a}{2}}v\|^2+\beta\|\gamma A^{\frac{1}{2}}v-A^{\frac{1}{2}}p\|^2\leq \varepsilon(\|A^{\frac{1}{2}}v\|^2+\|A^{\frac{1}{2}}p\|^2)+k_\varepsilon \|F\|_{\mathcal{H}}^2.
$$
Simple calculation shows that the inequality $\|A^{\frac{1}{2}}v\|^2+\|A^{\frac{1}{2}}p\|^2\leq \delta(\alpha_1\|A^{\frac{1}{2}}v\|^2+\beta\|\gamma A^{\frac{1}{2}}v-A^{\frac{1}{2}}p\|^2)$, where $\delta=\max\{\frac{\gamma^2+1}{\alpha_1},\frac{1}{\beta}\}$. Then the equivalence of norms  \eqref{norm-equi-2} and choosing $\varepsilon$ small enough imply that $\|(v,u,p,q,\eta)\|_{\mathcal{H}}\leq c\|F\|_{\mathcal{H}},~c>0$. Hence $0\in \varrho(\mathcal{A})$.  The proof is completed.
\end{proof}

\section{Explicit  decay rates}
\noindent

This section is mainly devoted to discussing the stability of system $(\ref{alpha-sys2})$.
{We shall prove that  the semigroup corresponding to system (\ref{eq2.8}) decays   polynomially under smooth initial states by frequency domain analysis.  Some new frequency multipliers are developed to help us give an delicate estimate for the  norm of the resolvent operator along the imaginary axis.} Let us first present our main result.


\begin{theorem}\label{alpha-t-4-2}
	Assume that the assumptions (A1)-(A2) hold. Then the semigroup $e^{\mathcal{A}t}$ associated with system (\ref{alpha-sys1})
is polynomially stable with decay rates $t^{-\frac{1}{2-2a}}$ under smooth initial conditions, that is, for the initial data $(v_{0},v_{1},p_{0}, p_{1}, w_0)\in\mathcal{D(\mathcal{A})}$, the solutions to system (\ref{alpha-sys1}) satisfy
	\begin{equation}\label{2.5}
	\|(v, v_{t},p,p_{t}, w)\|_{\mathcal{H}}\leq{C}{t^{-\frac{1}{2-2a}}}\|(v_{0},v_{1},p_{0}, p_{1}, w_0)\|_{D(\mathcal{A})},\quad \forall t\geq 1,
	\end{equation}
	where  $C>0$ is some constant.
\end{theorem}

In order to show Theorem \ref{alpha-t-4-2}, let us recall the following frequency characteristics on the polynomial stability  of bounded $C_0$-semigroups on Hilbert spaces (see \cite{BCT,bori}, and also see \cite{Roz} for generalized case and \cite{liuz} for a weaker result).

{\begin{remark}
	By Theorem \ref{alpha-t-4-2},  it is easy to see that as the parameter $a\in [0,1)$ becomes bigger, the decay rate $t^{-\frac{1}{2-2a}}$ becomes faster. Especially, when $a\to 1^-$,  we get that $-\frac{1}{2-2a}\to -\infty$.  Thus, this is consistent with the exponential decay  result \cite{zhangxuhan1} of the same system with viscoelastic damping ($a=1$, $A=-\Delta$). Indeed, by a detailed spectral analysis given in the next section, we can verify the optimality of this decay rate $t^{-\frac{1}{2-2a}}$ for the system under certain condition.
	\end{remark}}


{\begin{lemma}\label{alpha-l-4-2}
Let $\mathcal{A}$ be the infinitesimal generator of  a bounded $C_0$ semigroup $e^{\mathcal{A}t}$ on a Hilbert space $\mathcal{H}$ and satisfies $i\mathbb{R}\subset\rho(\mathcal{A})$. Then for a fixed $\omega>0$, $e^{\mathcal{A}t}$ is polynomially stable of order $\frac{1}{\omega}$, if and only if the following condition holds:
\begin{equation}\label{alpha-404}
 \sup_\lambda\limits\{|\lambda|^{-\omega}\|(i\lambda-\mathcal{A})^{-1}\|_{\mathcal{H}}\mid\lambda\in \mathbb{R}\}<\infty.
 \end{equation}
\end{lemma}
}

\noindent{\bf Proof of Theorem \ref{alpha-t-4-2}.~}
{From Lemma \ref{alpha-l-4-2}, it only suffices to verify that the conditions $i\mathbb{R}\subset\varrho(\mathcal{A})$ and \eqref{alpha-404} hold. The proof can be divided into two parts.
	
	\vspace{0.5cm}
\noindent{\bf Part I.  We show that $i\mathbb{R}\subset\varrho(\mathcal{A})$.}

 By contradiction,  suppose that $i\mathbb{R}\nsubseteq\varrho(\mathcal{A})$.
Thus, due to $0\in \varrho(\mathcal{A})$ and $\varrho(\mathcal{A})$ is an open set, we have
$$
0<\hat \lambda<\infty,
$$
in which
$$
\hat\lambda:=\sup\{R>0\; :\; [-Ri,Ri]\subset \varrho(\mathcal{A})\}.
$$
 Thanks to Banach-Steinhaus theorem,  there exist  sequences $X_n=(v_n,u_n,p_n,q_n,\eta_n)$ $\in \mathcal{D}(\mathcal{A})$ with $\|X_n\|_{\mathcal{H}}=1$ and $|\lambda_n|<
\hat{\lambda},~\lambda_n\rightarrow \hat{\lambda}$ such that
\begin{equation}\label{bs1-3}
(i\lambda_n-\mathcal{A})X_n\equiv (f^1_n,f^2_n,z^1_n,z^2_n,\nu_n)\rightarrow 0~in~\mathcal{H},
\end{equation}
specifically,
\begin{align}
\label{alpha-e1-3}& i \lambda_n v_n-u_n= f^1_n\rightarrow 0,~\mbox{in}~D(A^{\frac{1}{2}}),\\
\label{alpha-e2-3}&i \lambda_nu_n-\frac{1}{\rho}(-\alpha A v_{n}+\gamma\beta A p_{n}+\zeta A^av_n-\int_{0}^{\infty}g(s)A^a \eta_n(s)ds)=f^2_n\rightarrow0,~\mbox{in}~H,\\
\label{alpha-e3-3} &i \lambda_n p_n-q_n=z^1_n \to 0,~\mbox{in}~D(A^{\frac{1}{2}}),\\
 \label{alpha-e4-3} & i \lambda_n q_n- \frac{1}{\mu}(-\beta A p_{n}+\gamma\beta A v_{n})=z^2_n\rightarrow0,~\mbox{in}~H,\\
 \label{alpha-e5-3} &    i \lambda_n \eta_n-( u_n -\eta_{n,s})=\nu_n \rightarrow0,~\mbox{in}~L_{g}^2(\mathbb{R}_+,D(A^{\frac{a}{2}})).
\end{align}
In the sequel of this part, we aim to show
\begin{equation}\label{alpha-hh14}
\|\eta_n\|_{L_{g,a}^2},\; \|u_n\|,\; \|A^{\frac{1}{2}} v_{n}\|,\;\|A^{\frac{a}{2}} v_{n}\|,\; \|\gamma A^{\frac{1}{2}} v_{n}-A^{\frac{1}{2}} p_{n}\|,\; \|q_n\|=o(1).
\end{equation}

Note that the relationship between $\|A^{\frac{1}{2}} v_{n}\|$ and $\|A^{\frac{a}{2}} v_{n}\|$ in \eqref{norm-equi-1} implies that {it is sufficient to estimate only $\|A^{\frac{1}{2}} v_{n}\|$.}
Once we have shown \eqref{alpha-hh14}, we can directly obtain  $\|X_n\|_{\mathcal{H}}=o(1)$ from (\ref{inner-product}). Thus
contradicts the fact that $\|X_{n}\|_{\mathcal{H}}= 1$, thereby completing the proof of {\bf Part I}.

 For this aim, we divide it by five steps.

\vspace{0.3cm}
\noindent{\it Step 1. $\|\eta_n\|_{L_{g,a}^2}\rightarrow0,\quad n\to \infty.$}

By virtue of the dissipativeness of $\mathcal{A}$ (see  (\ref{alpha-E})) and $(\ref{bs1-3})$, we get
\begin{align}
\label{alpha-di3-3} \int_{0}^{\infty}g'(s)\|A^{\frac{a}{2}}\eta_{n}(s)\|^2ds&\rightarrow0.
\end{align}

Due to the assumption (A2), it holds that
\begin{eqnarray*}
0\leftarrow-\int_{0}^{\infty}g'(s)\|A^{\frac{a}{2}}\eta_{n}(s)\|^2ds\geq k_1\int_{0}^{\infty}g(s)\|A^{\frac{a}{2}}\eta_{n}(s)\|^2ds\geq 0,
\end{eqnarray*}
which implies
\begin{align}
\label{alpha-eta-11} \|\eta_n\|_{L_{g,a}^2}\rightarrow0.
\end{align}

\vspace{0.3cm}
\noindent{\it Step 2. $ \|u_n\|\rightarrow0,\quad n\to\infty.$}

{
Taking the $L_{g,a}^2-$inner product of equation $(\ref{alpha-e5-3})$ with $g(s)u_n$ yields
\begin{align}
\label{alpha-e5e} \langle i \lambda_n \eta_n,gu_n\rangle_{L_{g,a}^2}-\langle \eta_{n,s},gu_n\rangle_{L_{g,a}^2}+\langle gu_n,gu_n\rangle_{L_{g,a}^2}\rightarrow0.
\end{align}

We now estimate each term in (\ref{alpha-e5e}).

{\it Observation I.} Using Cauchy-Schwartz inequality and (\ref{alpha-eta-11}), along with $|\lambda_n|<
\hat{\lambda}$, we  have
\begin{eqnarray}\nonumber
|\langle i \lambda_n \eta_n,gu_n\rangle_{L_{g,a}^2}|&=&|\lambda_n|\big|\int_{0}^{\infty}g(s)\langle A^{\frac{a}{2}}\eta_{n}(s), A^{\frac{a}{2}}u_n\rangle ds\big|\cr
&\leq&{|\lambda_n|\|A^{\frac{a}{2}}u_n\|\int_{0}^{\infty}g(s)\|A^{\frac{a}{2}}\eta_n(s)\|ds}\cr
&\leq&|\lambda_n|\|A^{\frac{a}{2}}u_n\|\big(\int_{0}^{\infty}g(s)ds\big)^{\frac{1}{2}}\big(\int_{0}^{\infty}g(s)\|A^{\frac{a}{2}}\eta_n(s)\|^2 ds\big)^{\frac{1}{2}}\cr
&=&|\lambda_n|\|A^{\frac{a}{2}}u_n\|\sqrt{\zeta}\|\eta_n\|_{L_{g,a}^2}\rightarrow0.
\end{eqnarray}
Herein we have used the boundedness of $\|A^{\frac{a}{2}}u_n\|$ due to \eqref{norm-equi-1}, \eqref{alpha-e1-3} and the boundedness of $\|A^{\frac{1}{2}}v_n\|$.

{\it Observation II.} Thanks to \eqref{alpha-E} and \eqref{alpha-eta-11},   applying Cauchy-Schwartz inequality {along with the assumption  (A1)}, yields that
{\begin{eqnarray}\nonumber
|\langle \eta_{n,s},gu_n\rangle_{L_{g,a}^2}|&=&\big|\int_{0}^{\infty}g(s)\langle A^{\frac{a}{2}}\eta_{n,s}(s),A^{\frac{a}{2}}u_n\rangle ds\big|\cr
&=&\big|-\int_{0}^{\infty}g'(s)\langle A^{\frac{a}{2}}\eta_{n}(s), A^{\frac{a}{2}}u_n\rangle ds\big|\cr
&\leq&k_0\big|\int_{0}^{\infty}g(s)\langle A^{\frac{a}{2}}\eta_{n}(s), A^{\frac{a}{2}}u_n\rangle ds\big|\cr
&\leq&k_0\|A^{\frac{a}{2}}u_n\|\int_{0}^{\infty}g(s)\|A^{\frac{a}{2}}\eta_n(s)\| ds\cr
&\leq&k_0\|A^{\frac{a}{2}}u_n\|\big(\int_{0}^{\infty}g(s)ds\big)^{\frac{1}{2}}\big(\int_{0}^{\infty}g(s)\|A^{\frac{a}{2}}\eta_n(s)\|^2ds\big)^{\frac{1}{2}}\cr
&=&{k_0\|A^{\frac{a}{2}}u_n\|\sqrt{\zeta}\|\eta_n\|_{L_{g,a}^2}\rightarrow0.}
\end{eqnarray}}

Substituting {\it Observation I and II} into $(\ref{alpha-e5e})$ yields that $\langle gu_n,gu_n\rangle_{L_{g,\alpha}^2}\rightarrow0$.  Note that
\begin{eqnarray}\nonumber
0\leftarrow|\langle gu_n,gu_n\rangle_{L_{g,a}^2}|&=&{\int_{0}^{\infty}g(s)\|A^{\frac{a}{2}}u_n\|^2 ds=\zeta\|A^{\frac{a}{2}}u_n\|^2},
\end{eqnarray}
{here we have used the assumption  (A1). Thus,}
\begin{align}
\nonumber
 \|A^{\frac{a}{2}}u_n\|\rightarrow0.
\end{align}
Hence, by virtue of the inequality $\|u_n\|\leq\|A^{\frac{a}{2}}u_n\|$, it gives that
\begin{align}
\label{un-alpha-2} \|u_n\|\rightarrow0.
\end{align}
}
\noindent{\it Step 3. $\|A^{\frac{1}{2}} v_{n}\|\to 0,\quad n\to \infty$.}

Substituting $(\ref{alpha-e1-3})$ and $(\ref{alpha-e3-3})$ into $(\ref{alpha-e2-3})$ and $(\ref{alpha-e4-3})$, respectively, {yields}
\begin{align}
\label{alpha-ne2} -\lambda^2_n\rho v_n-(-\alpha A v_{n}+\gamma\beta A p_{n}+\zeta A^av_n-\int_{0}^{\infty}g(s)A^a \eta_n(s)ds)\rightarrow0,~\mbox{in}~H,&\\
\label{alpha-ne4}  -\lambda^2_n \mu p_n-(-\beta A p_{n}+\gamma\beta A v_{n})\rightarrow0,~\mbox{in}~H.&
\end{align}

By $\alpha_1=\alpha-\gamma^2\beta>0$, $(\ref{alpha-ne2})$ and $(\ref{alpha-ne4})$, we get
\begin{align}
\nonumber-\lambda^2_n\rho v_n-[-\alpha_1Av_{n}+\zeta A^av_n-\gamma\beta(\gamma A v_{n}-A p_{n})-\int_{0}^{\infty}g(s)A^a \eta_n(s)ds]&\\
\label{alpha-ne21}\rightarrow0,~\mbox{in}~H,&\\
\label{alpha-ne41} -\lambda^2_n\gamma \mu p_n-\gamma\beta( \gamma A v_{n}-A p_{n})\rightarrow0,~\mbox{in}~H.&
\end{align}

Thus, by combining (\ref{alpha-ne21}) with (\ref{alpha-ne41}) and eliminating the common item\\ $\gamma\beta( \gamma A v_{n}-A p_{n})$, we obtain
\begin{align}
\label{alpha-ne2ne4} &\lambda^2_n\rho v_n+\lambda^2_n\gamma \mu p_n-\alpha_1 A v_{n}+\zeta A^av_n-\int_{0}^{\infty}g(s)A^a \eta_n(s)ds\rightarrow0,~~\mbox{in}~H.
\end{align}

Then, taking inner product of $(\ref{alpha-ne2ne4})$ with $v_n$, we have
\begin{align}
\nonumber &\lambda^2_n\rho \|v_n\|^2+\lambda^2_n\gamma \mu \langle p_n,v_n\rangle-\alpha_1 \|A^{\frac{1}{2}} v_{n}\|^2+\zeta \|A^{\frac{a}{2}} v_{n}\|^2\\
\label{alpha-ne2ne4-1}&-\langle\int_{0}^{\infty}g(s)A^\frac{a}{2}\eta_n(s)ds,A^\frac{a}{2}v_n\rangle\rightarrow0.
\end{align}

By utilizing $(\ref{alpha-e1-3})$, $(\ref{alpha-e3-3})$ and Cauchy-Schwartz inequality, we get that the second term in (\ref{alpha-ne2ne4-1}) satisfies
\begin{align}
\label{alpha-lam-p-v} |\lambda^2_n\gamma \mu \langle p_n,v_n\rangle|\sim |\gamma \mu\langle q_n,u_n\rangle|\leq\gamma\mu\|u_n\|\|q_n\|\rightarrow 0.
\end{align}
{Here} we have used (\ref{un-alpha-2}) and the boundedness of $\|q_n\|$.

By Cauchy-Schwartz inequality, we get that the last term in (\ref{alpha-ne2ne4-1}) satisfies
\begin{align}
\label{alpha-eta-v}\big|\langle\int_{0}^{\infty}g(s)A^\frac{a}{2} \eta_n(s)ds,A^\frac{a}{2}v_n\rangle\big|&\leq\|A^\frac{a}{2}  v_n\|\int_{0}^{\infty}g(s)\|A^\frac{a}{2}\eta_n(s)\|ds\cr
&{\leq \sqrt{\zeta}\|A^\frac{a}{2}  v_{n}\|\eta_n\|_{L_{g,a}^2}}\cr
&\to 0,
\end{align}
where the inequality \eqref{norm-equi-1} and the boundedness of $\|A^{\frac{a}{2}}v_n\|$ are utilized.
This together with \eqref{un-alpha-2}, $(\ref{alpha-ne2ne4-1})$ and $(\ref{alpha-lam-p-v})$ shows that
\begin{align}
\nonumber
\alpha_1 \|A^{\frac{1}{2}} v_{n}\|^2-\zeta\|A^{\frac{a}{2}} v_{n}\|^2\rightarrow0,
\end{align}
and thus, by  \eqref{norm-equi}, we have
\begin{align}
\label{alpha-vn-3} &\|A^{\frac{1}{2}} v_{n}\|\rightarrow0.
\end{align}
\vspace{0.3cm}
\noindent{\it Step 4. $\|q_n\|\to 0,\quad n\to \infty.$}
\vspace{0.2cm}

Taking  inner product of (\ref{alpha-ne2ne4}) with $p_n$, along with Cauchy-Schwartz inequality, yields
\begin{align}
&\lambda^2_n\rho\langle v_n,p_n\rangle+\lambda^2_n\gamma \mu \|p_n\|^2-\alpha_1 \langle A^{\frac{1}{2}} v_{n}, A^{\frac{1}{2}} p_{n}\rangle+\zeta\langle A^{\frac{a}{2}} v_{n}, A^{\frac{a}{2}}p_{n}\rangle\cr
\label{434}&-\langle\int_{0}^{\infty}g(s)A^\frac{a}{2} \eta_n(s)ds,A^\frac{a}{2}p_n\rangle\to0.
\end{align}

On the one hand, we claim that
\begin{equation}\label{alpha-435}
\lambda^2_n\rho \langle v_n,p_n\rangle=\rho \langle \lambda_n v_n,\lambda_np_n\rangle\to0.
\end{equation}
Indeed, by (\ref{alpha-e1-3}) and (\ref{un-alpha-2}), we know that $\|\lambda_nv_n\|\to 0$, which together with the H\"{o}lder inequality and (\ref{alpha-e3-3}), yields (\ref{alpha-435}).

On the other hand,  using H\"{o}lder inequality again, along with (\ref{alpha-di3-3}), \eqref{alpha-eta-11}, \eqref{alpha-vn-3} and the boundedness of $\|A^{\frac{1}{2}} p_n\|$, we obtain that the last three terms in (\ref{434}) satisfy
$$\langle A^{\frac{1}{2}} v_{n}, A^{\frac{1}{2}}p_{n}\rangle\to 0,~\langle A^{\frac{a}{2}} v_{n}, A^{\frac{a}{2}}p_{n}\rangle\to 0$$
and
$$\langle\int_{0}^{\infty}g(s)A^\frac{a}{2} \eta_n(s)ds,A^\frac{a}{2}p_n\rangle\to 0.$$

Therefore, by (\ref{434}), we get
\begin{equation}\nonumber
\lambda^2_n\gamma \mu \|p_n\|^2\to 0,
\end{equation}
which along with (\ref{alpha-e3-3})  and  $|\lambda_n|<
\hat{\lambda},~\lambda_n\rightarrow \hat{\lambda}$ leads to
\begin{equation}\label{alpha-qn+3}
\|q_n\|\to 0.
\end{equation}

\vspace{0.3cm}
\noindent{\it Step 5. $\|\gamma A^{\frac{1}{2}} v_{n}-A^{\frac{1}{2}} p_{n}\|,\quad n\to \infty.$}
\vspace{0.2cm}

Taking the inner product  of $(\ref{alpha-ne2})$ and $(\ref{alpha-ne4})$ with $v_n$ and $p_n$, respectively, then
\begin{align}
 \nonumber&-\lambda^2_n\rho\|v_n\|^2+\alpha\|A^{\frac{1}{2}} v_{n}\|^2+\zeta \|A^\frac{a}{2}v_n\|^2-\gamma\beta \langle A^{\frac{1}{2}}v_{n}, A^{\frac{1}{2}} p_{n}\rangle\\
\label{alpha-ne2i}&-\langle\int_{0}^{\infty}g(s)A^\frac{a}{2} \eta_n(s)ds,A^\frac{a}{2}v_n\rangle\rightarrow0,\\
\label{alpha-ne4i} &- \lambda^2_n \mu\|p_n\|^2+\beta\|A^{\frac{1}{2}} p_{n}\|^2-\gamma\beta\langle A^{\frac{1}{2}} v_{n}, A^{\frac{1}{2}}p_{n}\rangle\rightarrow0.
\end{align}

Adding $(\ref{alpha-ne2i})$ and $(\ref{alpha-ne4i})$, and  by virtue of $(\ref{alpha-e1-3})$, $(\ref{alpha-e3-3})$, \eqref{alpha-eta-v} and $\alpha_1=\alpha-\gamma^2\beta$, we obtain
{\begin{align}
\label{alpha-1} \alpha_1\|A^{\frac{1}{2}} v_{n}\|^2+\zeta\|A^\frac{a}{2}v_n\|^2-\rho\|u_n\|^2 +\beta\|\gamma A^{\frac{1}{2}} v_{n}-A^{\frac{1}{2}} p_{n}\|^2-\mu\|q_n\|^2 \rightarrow0.
\end{align}}

Substituting $(\ref{un-alpha-2})$, $(\ref{alpha-vn-3})$ and \eqref{alpha-qn+3} into $(\ref{alpha-1})$ gives that
\begin{align}
\label{alpha-qn} & \|\gamma A^{\frac{1}{2}} v_{n}-A^{\frac{1}{2}} p_{n}\|\rightarrow0,
\end{align}

Summing up the above five steps, we have derived  that $\|X_n\|_{\mathcal{H}}\rightarrow0$, which contradicts $\|X_n\|_{\mathcal{H}}=1$. Thus, the condition (I) in Lemma \ref{alpha-l-4-2} has been verified.

{\vspace{0.5cm}
\noindent{\bf Part II. We show  (\ref{alpha-404})  {with $\omega=2-2a$}  holds.}

The proof by contradiction is still employed.
 If it is not true, then thanks to Banach-Steinhaus theorem, there exist  sequences $X_n=(v_n,u_n,p_n,q_n,w_n)\in \mathcal{D}(\mathcal{A})$ with $\|X_n\|_{\mathcal{H}}=1$ and $\lambda_n\rightarrow \infty$ such that
\begin{equation}\label{alpha-bs1-4}
\lambda_n^{2-2a}(i\lambda_n-\mathcal{A})X_n\equiv (\tilde{f}^1_n,\tilde{f}^2_n,\tilde{z}^1_n,\tilde{z}^2_n,\tilde{\nu}_n)\rightarrow 0~\mbox{in}~\mathcal{H},
\end{equation}
that is,
\begin{align}
\label{alpha-e1-4} &\lambda_n^{2-2a}(i \lambda_n v_n-u_n)= \tilde{f}^1_n\rightarrow 0,~\mbox{in}~D(A^{\frac{1}{2}}),\\
\nonumber &\lambda_n^{2-2a}\big[i \lambda_nu_n-\frac{1}{\rho}(-\alpha A v_{n}+\gamma\beta A p_{n}+\zeta A^av_n-\int_{0}^{\infty}g(s)A^a \eta_n(s)ds)\big]=\tilde{f}^2_n\\
\label{alpha-e2-4}&~~~~~~~~~~~~~~~~~~~~~~~~~~~~~~~~~~~~~~~~~~~~~~~~~~~~~~~~~~~~~~~~~~~~~~~~~~~~~~~~~\rightarrow0,~\mbox{in}~H,\\
\label{alpha-e3-4} &\lambda_n^{2-2a}(i \lambda_n p_n-q_n)=\tilde{z}^1_n \rightarrow0,~\mbox{in}~D(A^{\frac{1}{2}}),\\
 \label{alpha-e4-4} &\lambda_n^{2-2a}( i \lambda_n q_n- \frac{1}{\mu}(-\beta A p_{n}+\gamma\beta A v_{n}))=\tilde{z}^2_n\rightarrow0,~\mbox{in}~H,\\
 \label{alpha-e5-4}&  \lambda_n^{2-2a}( i \lambda_n \eta_n-( u_n -\eta_{n,s}))=\tilde{\nu}_n\rightarrow0,~\mbox{in}~L_{g}^2(\mathbb{R}_+,D(A^{\frac{a}{2}})).
\end{align}
In the following, similar to {\bf Part I}, we still aim to show (\ref{alpha-hh14}) holds
 so as to obtain $\|X_n\|_{\mathcal{H}}=o(1)$, which induces a
contradiction.

Thanks to the dissipativeness of $\mathcal{A}$ (see  (\ref{alpha-E})) and \eqref{alpha-bs1-4}, we get
\begin{align}
\nonumber
 |\lambda_n|^{2-2a}\int_{0}^{\infty}g'(s)\|A^{\frac{a}{2}}\eta_{n}(s)\|^2ds&\rightarrow0.
\end{align}
Thus, due to the assumption (A2), we obtain that
\begin{eqnarray*}
0\leftarrow-|\lambda_n|^{2-2a}\int_{0}^{\infty}g'(s)\|A^{\frac{a}{2}}\eta_{n}(s)\|^2ds\geq k_1|\lambda_n|^{2-2a}\int_{0}^{\infty}g(s)\|A^{\frac{a}{2}}\eta_{n}(s)\|^2ds\geq 0,
\end{eqnarray*}
which implies that
\begin{align}
\nonumber
|\lambda_n|^{1-a}\|\eta_n\|_{L_{g,a}^2}\rightarrow0,
\end{align}
and hence,
\begin{align}
\label{alpha-w0-4-1} \|\eta_n\|_{L_{g,a}^2}= {\frac{1}{|\lambda_n|^{1-a}}o(1)}.
\end{align}

Multiplying \eqref{alpha-e5-4} by {$\lambda_n^{-(2-2a)}$}, we get
\begin{align}
\label{alpha-e5-out-5}     i \lambda_n \eta_n-( u_n -\eta_{n,s})&=\frac{1}{\lambda_n^{2-2a}}\tilde{\nu}_n\rightarrow 0,~\text{in}~L_{g}^2(\mathbb{R}_+,D(A^{\frac{a}{2}})).
\end{align}
{

Taking the $L_{g,a}^2-$inner product of $(\ref{alpha-e5-out-5})$ with $u_n$ yields
\begin{align}
\nonumber&\int_{0}^{\infty}i \lambda_n g(s)\langle A^{\frac{a}{2}}\eta_n, A^{\frac{a}{2}}u_n\rangle ds+ \int_{0}^{\infty}g(s)\langle A^{\frac{a}{2}}\eta_{n,s}, A^{\frac{a}{2}}u_n\rangle ds\\
\label{alpha-e5e-g-u-fir}&-\int_{0}^{\infty}g(s)\|A^{\frac{a}{2}}u_n\|^2ds
=\frac{1}{\lambda_n^{2-2a}}\int_{0}^{\infty} g(s)\langle A^{\frac{a}{2}}\tilde{\nu}_n, A^{\frac{a}{2}}u_n\rangle ds.
\end{align}

 We  estimate the first two terms on the left-hand side of (\ref{alpha-e5e-g-u-fir}).
Using Cauchy-Schwarz inequality and \eqref{alpha-w0-4-1} yields that the first term satisfies
{\begin{eqnarray}\label{alpha-et-u}
\big|\int_{0}^{\infty}i \lambda_n g(s)\langle A^{\frac{a}{2}}\eta_n, A^{\frac{a}{2}}u_n\rangle ds\big| &\leq&|\lambda_n|\sqrt{\zeta}\|\eta_n\|_{L_{g,a}^2}\|A^{\frac{a}{2}}u_n\|\cr
&=&|\lambda_n|^a\sqrt{\zeta}\|A^{\frac{a}{2}}u_n\|o(1).
\end{eqnarray}

For the second term of \eqref{alpha-e5e-g-u-fir}, thanks to (A2) and \eqref{alpha-w0-4-1}, applying the Cauchy-Schwartz inequality leads to
{\begin{eqnarray}\label{alpha-g-u-2}
&&\big|\int_{0}^{\infty}g(s)\langle A^{\frac{a}{2}}\eta_{n,s}, A^{\frac{a}{2}}u_n\rangle ds\big|\cr
&=&\big|-\int_{0}^{\infty}g'(s)\langle A^{\frac{a}{2}}\eta_{n}, A^{\frac{a}{2}}u_n\rangle ds\big|\cr
&\leq&k_0\bigg|\int_{0}^{\infty}g(s)\langle A^{\frac{a}{2}}\eta_{n}, A^{\frac{a}{2}}u_n\rangle ds\bigg|\cr
&\leq&k_0\|A^{\frac{a}{2}}u_n\|\int_{0}^{\infty}g(s)\|A^{\frac{a}{2}}\eta_n\|ds\cr
&\leq&k_0\|A^{\frac{a}{2}}u_n\|\sqrt{\zeta}\|\eta_n\|_{L_{g,a}^2}\cr
&=&\frac{1}{|\lambda_n|^{1-a}}k_0\sqrt{\zeta}\|A^{\frac{a}{2}}u_n\|o(1).
\end{eqnarray}}

Similarly, the right-hand side of (\ref{alpha-e5e-g-u-fir}) gives that
\begin{eqnarray}\label{right-hand}
\big|\frac{1}{\lambda_n^{2-2a}}\int_{0}^{\infty} g(s)\langle A^{\frac{a}{2}}\tilde{\nu}_n , A^{\frac{a}{2}}u_n\rangle ds\big|\leq\frac{1}{|\lambda_n|^{2-2a}}\sqrt{\zeta}\|\tilde{\nu}_n\|_{L_{g,a}^2}\|A^{\frac{a}{2}}u_n\|o(1).
\end{eqnarray}}

Substituting \eqref{alpha-et-u}, \eqref{alpha-g-u-2} and \eqref{right-hand} into $(\ref{alpha-e5e-g-u-fir})$, along with $\zeta=\int_{0}^{\infty}g(s)ds>0$, we get
\begin{eqnarray}\nonumber
\zeta\|A^{\frac{a}{2}}u_n\|^2&\leq&|\lambda_n|^a\sqrt{\zeta}\|A^{\frac{a}{2}}u_n\|o(1)+\frac{1}{|\lambda_n|^{1-a}}k_0\sqrt{\zeta}\|A^{\frac{a}{2}}u_n\|o(1)\cr
&&+\frac{1}{|\lambda_n|^{2-2a}}\sqrt{\zeta}\|\tilde{\nu}_n\|_{L_{g,a}^2}\|A^{\frac{a}{2}}u_n\|o(1).
\end{eqnarray}
which yields that
{\begin{align}
\label{alpha-un} \|A^{\frac{a}{2}}u_n\|= |\lambda_n|^ao(1).
\end{align}}
{Taking $L^2-$inner product of \eqref{alpha-e5-out-5} with $A^{a-1}u_n$ and then multiplying it with $g(s)$ and integrating the obtained identity from $0$ to $\infty$ with respect to $s$, we get}
\begin{align}
\nonumber\zeta\|A^{\frac{a-1}{2}}u_n\|^2&=\int_{0}^{\infty}i \lambda_n g(s)\langle A^{\frac{a-1}{2}}\eta_n, A^{\frac{a-1}{2}}u_n\rangle ds+\int_{0}^{\infty}g(s)\langle A^{\frac{a-1}{2}}\eta_{n,s}, A^{\frac{a-1}{2}}u_n\rangle ds\\
\label{alpha-e5e-g-u}&~~~-\frac{1}{\lambda_n^{2-2a}}\int_{0}^{\infty}g(s)\langle A^{\frac{a-1}{2}}\tilde{\nu}_n,A^{\frac{a-1}{2}}u_n\rangle ds.
\end{align}

Now, let us estimate those three terms on the right-hand of (\ref{alpha-e5e-g-u}) so as to get the estimate of $\|A^{\frac{a-1}{2}}u_n\|$.

Similar to the discussions in \eqref{alpha-g-u-2} and \eqref{right-hand}, using the Cauchy-Schwarz inequality, along with \eqref{alpha-w0-4-1}, we have that the second term on the right-hand of (\ref{alpha-e5e-g-u}) satisfies
\begin{align}
\label{alpha-e5e-g-u-1} &\int_{0}^{\infty}g(s)\langle A^{\frac{a-1}{2}}\eta_{n,s}, A^{\frac{a-1}{2}}u_n\rangle ds-\frac{1}{\lambda_n^{2-2a}}\int_{0}^{\infty}g(s)\langle A^{\frac{a-1}{2}}{\tilde{\nu}_n}, A^{\frac{a-1}{2}}u_n\rangle ds\cr
&\leq\frac{1}{|\lambda_n|^{1-a}}k_0\sqrt{\zeta}\|A^{\frac{a-1}{2}}u_n\|o(1)+\frac{1}{|\lambda_n|^{2-2a}}\sqrt{\zeta}\|\tilde{\nu}_n\|_{L_{g,a}^2}\|A^{\frac{a-1}{2}}u_n\|o(1).
\end{align}

In order to estimate the first term on the right-hand of \eqref{alpha-e5e-g-u}, {taking the inner product on $\eqref{alpha-e2-4}$ with $\lambda_n^{-(2-2a)}\int_{0}^{\infty}g(s)A^{a-1}\eta_nds$, we get
\begin{eqnarray}\nonumber
&&\big|\int_{0}^{\infty} g(s)\langle A^{\frac{a-1}{2}}\eta_n, A^{\frac{a-1}{2}}i \lambda_nu_n\rangle ds\big|\cr
&=&\big|\int_{0}^{\infty} g(s)\langle A^{a-1}\eta_n, i \lambda_nu_n\rangle ds\big|\cr
&\leq&\frac{1}{\rho}\bigg[\alpha\big|\int_{0}^{\infty} g(s)\langle A^{\frac{a}{2}}\eta_{n}, A^{\frac{a}{2}}v_n\rangle ds\big|+\gamma\beta\big|{\int_{0}^{\infty}g(s)\langle A^{\frac{a}{2}}\eta_n, A^{\frac{a}{2}} p_n\rangle ds}\big|\cr
&&+\zeta\big|\int_{0}^{\infty}g(s)\langle A^{\frac{2a-1}{2}}\eta_{n}, A^{\frac{2a-1}{2}}v_n\rangle ds\big|+\big|\int_{0}^{\infty}g(s)\|A^{\frac{2a-1}{2}}\eta_{n}\|^2ds\big|\cr
&&+\frac{1}{|\lambda_n|^{2-2a}}\big|\int_{0}^{\infty}g(s)\langle A^{\frac{2a-1}{2}}\eta_{n}, A^{\frac{2a-1}{2}}\tilde{f}_n^2\rangle ds\big|\bigg].
\end{eqnarray}
Thanks to $\frac{2a-1}{2}< \frac{a}{2}\;  \mbox{for}\; a\in[0,1)$, by the continuous embedding $D(A^\frac{a}{2})\hookrightarrow D(A^{\frac{2a-1}{2}})$, we know that $\|A^{\frac{2a-1}{2}}\hat{w}\| \leq \tilde{\kappa}\|A^{\frac{a}{2}}\hat{w}\|,~\tilde{\kappa}>0,\forall \hat{w}\in D(A^{\frac{a}{2}})$. Therefore, we obtain that
\begin{eqnarray}\label{alpha-e5e-g-u-3}
&&\big|\int_{0}^{\infty} g(s)\langle A^{\frac{a-1}{2}}\eta_n, A^{\frac{a-1}{2}}i \lambda_nu_n\rangle ds\big|\cr
&\leq&\frac{1}{\rho}\bigg[\frac{1}{|\lambda_n|^{1-a}}\alpha\sqrt{\zeta}\|A^{\frac{a}{2}}v_n\|o(1)+\frac{1}{|\lambda_n|^{1-a}}\gamma\beta\sqrt{\zeta}\|A^{\frac{a}{2}}p_n\|o(1)\cr
&&+\frac{1}{|\lambda_n|^{1-a}}\tilde{\kappa}^2\zeta^{\frac{3}{2}}\|A^{\frac{a}{2}}v_n\|o(1)+\tilde{\kappa}^2\frac{1}{|\lambda_n|^{2-2a}}o(1)\cr
&&+\frac{1}{|\lambda_n|^{3-3a}}\tilde{\kappa}^2\sqrt{\zeta}\|A^{\frac{a}{2}}\tilde{f}_n^2\|o(1)\bigg]\rightarrow0,
\end{eqnarray}
where the boundedness of $\|A^{\frac{a}{2}}v_n\|$, $\|A^{\frac{a}{2}}p_n\|$ and $\|A^{\frac{a}{2}}\tilde{f}_n^2\|$ are used here.
Therefore, substituting \eqref{alpha-e5e-g-u-1}, \eqref{alpha-e5e-g-u-3} into \eqref{alpha-e5e-g-u}, we have
\begin{align}
\label{alpha-un-1} {\|A^{\frac{a-1}{2}}u_n\|=\frac{1}{|\lambda_n|^{1-a}}o(1).}
\end{align}
}
%

Thus, by interpolation, we get
$$\|u_n\|\leq \|A^{\frac{a}{2}}u_n\|^{1-a}\|A^{\frac{a-1}{2}}u_n\|^a,$$
which along with \eqref{alpha-un} and \eqref{alpha-un-1}, yields
\begin{align}
\label{alpha-un-2} \|u_n\|\to 0.
\end{align}

Thus, by the same discussion as given in  {\it Step 3}--{\it Step 5}  in {\bf Part I}, it still holds that
\begin{align}
\nonumber
\|A^{\frac{1}{2}} v_{n}\|,\;\|A^{\frac{a}{2}} v_{n}\|,\; \|\gamma A^{\frac{1}{2}} v_{n}-A^{\frac{1}{2}} p_{n}\|,\; \|q_n\|\to 0,\quad n\to \infty,
\end{align}
which together with \eqref{alpha-w0-4-1} and {(\ref{alpha-un})} yields that
 $\|X_n\|_{\mathcal{H}}\rightarrow0$, as $n\to \infty$, which contradicts $\|X_n\|_{\mathcal{H}}=1$. Hence, the condition (\ref{alpha-404}) holds.

Therefore, by {\bf Part I}, {\bf II} along with Lemma \ref{alpha-l-4-2}, the result in Theorem \ref{alpha-t-4-2} holds. The proof is completed.}}
}

\section{Eigenvalue asymptotics and sharpness of obtained decay rates}
\noindent

This section is devoted to presenting a  careful spectral analysis for the system operator $\mathcal{A}$ with
the exponential-type kernel function. Specifically,  we obtain the asymptotic expressions of the eigenvalues of  $\mathcal{A}$ given as follows.
{\begin{theorem}\label{alpha-p-4-2}
Assume that  the memory kernel function is of the exponential form $g(s)=e^{-\delta s}$, where $\delta$ is a positive constant. Then the asymptotic expressions of the spectrum of $\mathcal{A}$ are given as follows:
\begin{align}
\label{alpha-lamada-4-0}&{\lambda_{k,0}=-\delta+\frac{1}{\alpha_1}\frac{1}{\xi_k^{1-a}}+{O(\frac{1}{{\xi_k}^{2-a}})}},\\
\label{alpha-lamada-4-23}\lambda_{k,j,\pm}&=-\frac{\hat{m}_j}{m_j}\frac{1}{\xi_k^{1-a}}\pm m_j^\frac{1}{2}\sqrt{\xi_k}i+O(\frac{1}{\xi_k^{\frac{3}{2}-a}}),~a\in [0,1),~j=1,2,
\end{align}
where $\xi_k,\; k=1,2,\cdots$ are the eigenvalues of $A$ as given in (\ref{egienA}), and the constants $m_j,\hat{m}_j,\; j=1,2$ are given by
\begin{equation}\label{mj}
\left\{
\begin{array}{ll}
m_j:=\frac{(\frac{\beta}{\mu}+\frac{\alpha}{\rho})+(-1)^j\sqrt{(\frac{\beta}{\mu}+\frac{\alpha}{\rho})^2-\frac{4\alpha_1\beta}{\rho\mu}}}{2},\\
\hat{m}_j:=\frac{1}{2}\left(1+(-1)^{j-1}\frac{(\frac{\beta}{\mu}-\frac{\alpha}{\rho})}{\sqrt{(\frac{\beta}{\mu}+\frac{\alpha}{\rho})^2-\frac{4\alpha_1\beta}{\rho\mu}}}\right).
\end{array}
\right.
\end{equation}
\end{theorem}}

\begin{proof} Note that the assumptions (A1) and (A2) are all fulfilled under the choice of $g(s)=e^{-\delta s}$. Consider the eigenvalue problem

\[(\lambda I-\mathcal{A})(v,u,p,q,\eta)^T=0,~~\text{for}~(v,u,p,q,\eta)^T\in \mathcal{D}(\mathcal{A}),
\]
that is,
\begin{eqnarray}\label{alpha-eigen}
\left\{
\begin{array}{l}
             \lambda^2 v-\frac{1}{\rho}[-\alpha A v+\gamma\beta A p+\zeta A^a v-\int_{0}^{\infty}e^{-\delta s}A^a \eta(s)ds]=0,\\
             \lambda^2 p-\frac{1}{\mu}[-\beta A p+\gamma\beta A v]=0,\\
             \lambda \eta+\eta_s-\lambda v=0,\\
{\eta(0)}=0.
\end{array}
\right.
\end{eqnarray}

Solving the equation of $\eta$ and by applying the condition ${\eta(0)}=0$, it follows that
\begin{eqnarray}\label{alpha-eta-1}
\eta(s)=(1-e^{-\lambda s})v,
\end{eqnarray}
and then substituting \eqref{alpha-eta-1} and $\zeta:=\int_{0}^{\infty}e^{-\delta s}ds=\frac{1}{\delta}$ into {equations $(\ref{alpha-eigen})$,} when $\lambda\neq -\delta$, we have
\begin{eqnarray}\label{alpha-eigen-m}
\left\{
\begin{array}{l}
             \lambda^2 v-\frac{1}{\rho}[-\alpha A v+\gamma\beta A p{+}\frac{ 1}{\lambda+\delta}A^a v]=0,\\
             \lambda^2 p-\frac{1}{\mu}(-\beta A p+\gamma\beta A v)=0.
\end{array}
\right.
\end{eqnarray}

Recall that $\Psi_k$ is the eigenfunction of $A$ corresponding to $\xi_k$ and $A\Psi_k=\xi_k\Psi_k$. Set
\begin{equation}\label{alpha-v-p-k}
\left\{
\begin{array}{ll}
v_k=C_v \Psi_k,\\
p_k=C_p \Psi_k.\\
\end{array}
\right.
\end{equation}
Substituting $(\ref{alpha-v-p-k})$ into $(\ref{alpha-eigen-m})$ yields that
\begin{eqnarray}\nonumber
\left(
\begin{array}{cc}
  \lambda^2+\frac{\alpha{ \xi_k}}{\rho}-\frac{\xi_k^a}{\rho(\lambda+\delta)} & -\frac{\gamma\beta\xi_k}{\rho}\\
  -\frac{\gamma\beta{ \xi_k}}{\mu} & \lambda^2+\frac{\beta\xi_k}{\mu}
  \end{array}
  \right)\left(
           \begin{array}{c}
             C_v \\
             C_p \\
           \end{array}
         \right)\Psi_k=0.
\end{eqnarray}
Then $\lambda$ is an eigenvalue of $\mathcal{A}$ if and only if
\begin{eqnarray}\nonumber
\Delta(\lambda):=\det\left(
\begin{array}{cc}
  \lambda^2+\frac{\alpha{ \xi_k}}{\rho}-\frac{\xi_k^a}{\rho(\lambda+\delta)} & -\frac{\gamma\beta\xi_k}{\rho}\\
  -\frac{\gamma\beta{ \xi_k}}{\mu} & \lambda^2+\frac{\beta\xi_k}{\mu}
  \end{array}
  \right)=0.
\end{eqnarray}
A direct calculation shows that
\begin{align}\label{alpha-eigen-dd1}
&\Delta(\lambda):=\lambda^4+\big[(\frac{\beta}{\mu}+\frac{\alpha}{\rho})\xi_k-\frac{\xi_k^a}{\rho(\lambda+\delta)}\big]\lambda^2+\frac{\alpha_1\beta \xi_k^2}{\rho\mu}-\frac{\beta\xi_k^{a+1}}{\rho\mu(\lambda+\delta)}=0.
\end{align}
Thus,  we get from  \eqref{alpha-eigen-dd1} that
\begin{eqnarray}\label{alpha-lam-ddd1}
{\lambda^2}&=&-\frac{1}{2}[(\frac{\beta}{\mu}+\frac{\alpha}{\rho})\xi_k-\frac{\xi_k^a}{\rho(\lambda+\delta)}]\cr
&&\pm\frac{1}{2}\sqrt{[(\frac{\beta}{\mu}+\frac{\alpha}{\rho})\xi_k-\frac{\xi_k^a}{\rho(\lambda+\delta)}]^2-\frac{4\alpha_1\beta\xi_k^2}{\rho\mu}+\frac{4\beta\xi_k^{a+1}}{\rho\mu(\lambda+\delta)}}.
\end{eqnarray}
{By appropriate deformation, \eqref{alpha-lam-ddd1} can be written as
\begin{eqnarray}\nonumber
{\lambda^2}&=&-\frac{(\frac{\beta}{\mu}+\frac{\alpha}{\rho})\mp\sqrt{(\frac{\beta}{\mu}+\frac{\alpha}{\rho})^2-\frac{4\alpha_1\beta}{\rho\mu}}}{2}\xi_k+\frac{1}{2}\bigg(1\pm\frac{\frac{\beta}{\mu}-\frac{\alpha}{\rho}}{\sqrt{(\frac{\beta}{\mu}+\frac{\alpha}{\rho})^2-\frac{4\alpha_1\beta}{\rho\mu}}}\bigg)\frac{1}{\rho(\lambda+\delta)}\xi_k^a\cr
&&\pm\frac{1}{2}\sqrt{(\frac{\beta}{\mu}+\frac{\alpha}{\rho})^2-\frac{4\alpha_1\beta}{\rho\mu}}\xi_k\left(\begin{array}{c}
                                                                                                                                                                                                1+\frac{2(\frac{\beta}{\mu}-\frac{\alpha}{\rho})}{[(\frac{\beta}{\mu}+\frac{\alpha}{\rho})^2-\frac{4\alpha_1\beta}{\rho\mu}]\rho(\lambda+\delta)\xi_k^{1-a}}\\
                                                                                                                                                                                                +\frac{1}{[(\frac{\beta}{\mu}+\frac{\alpha}{\rho})^2-\frac{4\alpha_1\beta}{\rho\mu}]\rho^2(\lambda+\delta)^2\xi_k^{2-2a}}
                                                                                                                                                                                            \end{array}\right)^{\frac{1}{2}}\cr
&&\mp\frac{1}{2}\sqrt{(\frac{\beta}{\mu}+\frac{\alpha}{\rho})^2-\frac{4\alpha_1\beta}{\rho\mu}}\xi_k\mp\frac{\frac{\beta}{\mu}-\frac{\alpha}{\rho}}{2\sqrt{(\frac{\beta}{\mu}+\frac{\alpha}{\rho})^2-\frac{4\alpha_1\beta}{\rho\mu}}\rho(\lambda+\delta)}\xi_k^{a}.
\end{eqnarray}
Define
\begin{eqnarray}\label{g1}
g_{1\pm}(\lambda)&:=&{\lambda^2}+\frac{(\frac{\beta}{\mu}+\frac{\alpha}{\rho})
\mp\sqrt{(\frac{\beta}{\mu}+\frac{\alpha}{\rho})^2-\frac{4\alpha_1\beta}{\rho\mu}}}{2}\xi_k\cr
&&-\frac{1}{2}\bigg(1\pm\frac{\frac{\beta}{\mu}-\frac{\alpha}{\rho}}{\sqrt{(\frac{\beta}{\mu}+\frac{\alpha}{\rho})^2-\frac{4\alpha_1\beta}{\rho\mu}}}\bigg)\frac{1}{\rho(\lambda+\delta)}\xi_k^a,
\end{eqnarray}
\begin{eqnarray}\label{g2}
g_{2\pm}(\lambda):&=&\pm\frac{1}{2}\sqrt{(\frac{\beta}{\mu}+\frac{\alpha}{\rho})^2-\frac{4\alpha_1\beta}{\rho\mu}}\xi_k\left(\begin{array}{c}
                                                                                                                                                                                                1+\frac{2(\frac{\beta}{\mu}-\frac{\alpha}{\rho})}{[(\frac{\beta}{\mu}+\frac{\alpha}{\rho})^2-\frac{4\alpha_1\beta}{\rho\mu}]\rho(\lambda+\delta)\xi_k^{1-a}}\\
                                                                                                                                                                                                +\frac{1}{[(\frac{\beta}{\mu}+\frac{\alpha}{\rho})^2-\frac{4\alpha_1\beta}{\rho\mu}]\rho^2(\lambda+\delta)^2\xi_k^{2-2a}}
                                                                                                                                                                                            \end{array}\right)^{\frac{1}{2}}\cr
&&\mp\frac{1}{2}\sqrt{(\frac{\beta}{\mu}+\frac{\alpha}{\rho})^2-\frac{4\alpha_1\beta}{\rho\mu}}\xi_k\mp\frac{\frac{\beta}{\mu}-\frac{\alpha}{\rho}}{2\sqrt{(\frac{\beta}{\mu}+\frac{\alpha}{\rho})^2-\frac{4\alpha_1\beta}{\rho\mu}}\rho(\lambda+\delta)}\xi_k^{a}.
\end{eqnarray}
Let $m_j,\; \hat{m}_j,\; j=1,2$ be defined as in (\ref{mj}).
Obviously, $m_j,\hat{m}_j$ are positive constants.
In what follows, the subscript $j$ always equals to $1,2$, and  we omit ``$j=1,2$" for brevity.

Now, let us
first calculate the roots of the equation $g_{1\pm}(\tilde{\lambda})=0$, that is
\begin{eqnarray}\nonumber
{\tilde{\lambda}^2}+m_j\xi_k-\frac{\hat{m}_j}{\tilde{\lambda}+\delta}\xi_k^a=0,
\end{eqnarray}
which can be transformed into the following cubic equation
\begin{eqnarray}\label{alpha-lam-dd4}
{\tilde{\lambda}^3}+\delta\tilde{\lambda}^2+m_j\xi_k\tilde{\lambda}+m_j\delta\xi_k-\hat{m}_j\xi_k^a=0.
\end{eqnarray}
Let
\begin{align}
\nonumber
 &\hat{p}_k:=9m_j\xi_k-3\delta^2=9m_j\xi_k+O(1),\\
\nonumber
&\hat{q}_k:=2\delta^3+18m_j\xi_k\delta-27\hat{m}_j\xi_k^a=\hat{q}_k=18m_j\xi_k\delta+O(\xi_k^a).
\end{align}
Define $\Lambda_k$ and $\Phi_{k,\pm}$ by
$$
\Lambda_k:=(\frac{\hat{q}_k}{2})^2+(\frac{\hat{p}_k}{3})^3,\quad \Phi_{k,\pm}:=-\frac{\hat{q}_k}{2}\pm \sqrt{\Lambda_k}.
$$
Then,  we can calculate directly that
\begin{eqnarray*}
\Lambda_k
&=&(\delta^3+9m_j\xi_k\delta-\frac{27}{2}\hat{m}_j\xi_k^a)^2+(3m_j\xi_k-\delta^2)^3
=27m_j^3\xi_k^3+O(\xi_k^2),\\
\Phi_{k,\pm}
&=&-9m_j\xi_k\delta\pm 3\sqrt{3}m_j^{\frac{3}{2}}\xi_k^{\frac{3}{2}}+O(\xi_k)=\pm 3\sqrt{3}m_j^\frac{3}{2}\xi_k^{\frac{3}{2}}+O(\xi_k),
\end{eqnarray*}
and hence
\begin{eqnarray*}
\Phi_{k,+}^{\frac{1}{3}}-\Phi_{k,-}^{\frac{1}{3}}&=&\frac{\Phi_{k,+}-\Phi_{k,-}}{\Phi_{k,+}^{\frac{2}{3}}+\Phi_{k,+}^{\frac{1}{3}}\Phi_{k,-}^{\frac{1}{3}}+\Phi_{k,-}^{\frac{2}{3}}}=\frac{2\sqrt{\Lambda_k}}{\Phi_{k,+}^{\frac{2}{3}}+\Phi_{k,+}^{\frac{1}{3}}\Phi_{k,-}^{\frac{1}{3}}+\Phi_{k,-}^{\frac{2}{3}}}\cr
&=&\frac{6\sqrt{3}m_j^\frac{3}{2}\xi_k^{\frac{3}{2}}+O(\xi_k)}{3m_j\xi_k+O(\xi_k^{\frac{2}{3}})}=2\sqrt{3}m_j^\frac{1}{2}\xi_k^{\frac{1}{2}}+O(1).
\end{eqnarray*}
Similarly,
\begin{eqnarray*}
\Phi_{k,+}^{\frac{1}{3}}+\Phi_{k,-}^{\frac{1}{3}}&=&\frac{\Phi_{k,+}+\Phi_{k,-}}{\Phi_{k,+}^{\frac{2}{3}}-\Phi_{k,+}^{\frac{1}{3}}\Phi_{k,-}^{\frac{1}{3}}+\Phi_{k,-}^{\frac{2}{3}}}=\frac{-\hat{q}_k}{\Phi_{k,+}^{\frac{2}{3}}-\Phi_{k,+}^{\frac{1}{3}}\Phi_{k,-}^{\frac{1}{3}}+\Phi_{k,-}^{\frac{2}{3}}}\cr
&=&\frac{-18m_j\xi_k\delta+O(\xi_k^a)}{9m_j\xi_k+O(\xi_k^{\frac{2}{3}})}=-2\delta+O(\frac{1}{\xi_k^{1-a}}).
\end{eqnarray*}

According to the well-known Cardano's formula, we calculate directly that the cubic equation \eqref{alpha-lam-dd4} admits the roots given as follows:
 for $k=1,2,\cdots$,
\begin{eqnarray}
\label{alpha-lamada-2-0}\tilde{\lambda}_{k,0}&=&\frac{1}{3}(\Phi_{k,+}^{\frac{1}{3}}+\Phi_{k,-}^{\frac{1}{3}}-\delta)=-\delta+O(\frac{1}{\xi_k^{1-a}}),\\
\tilde{\lambda}_{k,j,\pm}&=&-\frac{1}{3}(\Phi_{k,+}^{\frac{1}{3}}+\Phi_{k,-}^{\frac{1}{3}}+2\delta)\pm \frac{\sqrt{3}}{6}i(\Phi_{k,+}^{\frac{1}{3}}-\Phi_{k,-}^{\frac{1}{3}})\cr
\label{alpha-lamada-2-23}&&=\pm im_j^\frac{1}{2}\xi_k^\frac{1}{2}+O(1),\quad j=1,2.
\end{eqnarray}


{We first focus on giving a more ``precise"
expressions for $\tilde{\lambda}_{k,j,\pm}$ in (\ref{alpha-lamada-2-23}).}
Note that the real part of $\tilde{\lambda}_{k,j,\pm}$ is not explicitly expressed in (\ref{alpha-lamada-2-23}). We shall further get more details for it.
 Thanks to Vieta's formula, along with (\ref{alpha-lam-dd4}), the sum of the roots satisfies
\begin{align}
\label{alpha-lamada-vita}
&-\delta=\tilde{\lambda}_{k,0}+\tilde{\lambda}_{k,j,+}+\tilde{\lambda}_{k,j,-}=\tilde{\lambda}_{k,0}+Y_{k,0},
\end{align}
where $Y_{k,0}:=2\mathfrak{R}e\tilde{\lambda}_{k,j,\pm}$.
Since $\tilde{\lambda}_{k,0}=-\delta-Y_{k,0}$ is a root of  equation \eqref{alpha-lam-dd4}, it gives that
\begin{eqnarray}\nonumber
{(-\delta-Y_{k,0})^3}+\delta(-\delta-Y_{k,0})^2+m_j\xi_k(-\delta-Y_{k,0})+m_j\delta\xi_k-\hat{m}_j\xi_k^a=0.
\end{eqnarray}
Then a direct calculation yields that $Y_{k,0}$ is a real root of the following equation
\begin{align}
\label{lam-dd4-vita-1}& Y^3+3\delta Y^2+(\delta^2+m_j\xi_k)Y+\hat{m}_j\xi_k^a=0.
\end{align}

Define $Y_{k,j,\pm}:=-\delta-\tilde{\lambda}_{k,j,\pm}$. {By \eqref{alpha-lamada-2-23}, we know
$$Y_{k,j,\pm}=-\delta\mp im_j^\frac{1}{2}\xi_k^\frac{1}{2}+O(1).$$ By the fact that $\tilde{\lambda}_{k,j,\pm}$ are the roots of equation \eqref{alpha-lam-dd4}, we also assert that $Y_{k,j,\pm}$ are complex roots of  \eqref{lam-dd4-vita-1}.}
 Thus, by applying Vieta's formula for  \eqref{lam-dd4-vita-1}, we get
$$
-\hat{m}_j\xi_k^a=Y_{k,0}Y_{k,j,+}Y_{k,j,-},
$$
which implies that
\begin{align}
\nonumber
2\mathfrak{R}e\tilde{\lambda}_{k,j,\pm}=Y_{k,0}=-\frac{\hat{m}_j}{m_j}\frac{1}{\xi_k^{1-a}}+{O(\frac{1}{\xi_k^{\frac{3}{2}-a}})}.
\end{align}
Therefore, we obtain that
\begin{align}
\label{lamada-3-23}\tilde{\lambda}_{k,j,\pm}&=-\frac{\hat{m}_j}{m_j}\frac{1}{\xi_k^{1-a}}\pm m_j^\frac{1}{2}\sqrt{\xi_k}i+{O(\frac{1}{\xi_k^{\frac{3}{2}-a}})}.
\end{align}
Let $\Upsilon_k=\left\{ \lambda \in \mathbb{C}\big| |\lambda-\tilde{\lambda}_{k,j,\pm}|=\frac{1}{\xi_k^{\frac{3}{2}-a}}\right\}$. Inspired by \cite{rao}, thanks to Rouch\'{e}'s theorem, $g_{1\pm}(\lambda)$ and $g_{1\pm}(\lambda)-g_{2\pm}(\lambda)$ have the same number of zeros in the interior of each $\Upsilon_k$ if
\begin{equation}\label{lam-12-abs}
|g_{1\pm}(\lambda)|>|g_{2\pm}(\lambda)|~~\text{on}~~\Upsilon_k,~~\text{for $k$ large enough}.
\end{equation}
In fact, for $\lambda\in \Upsilon_k$,
\begin{eqnarray}\label{g1-1}
|g_{1\pm}(\lambda)|=\frac{\hat{m}_j}{m_j^{\frac{1}{2}}}\frac{1}{\xi_k^{\frac{1}{2}-a}}+O(\frac{1}{\xi_k^{1-a}}).
\end{eqnarray}
Using Taylor's expansion along with \eqref{g2}, we get the following estimate:
\begin{eqnarray}\nonumber
g_{2\pm}(\lambda):&=&\mp\frac{1}{4\sqrt{(\frac{\beta}{\mu}+\frac{\alpha}{\rho})^2-\frac{4\alpha_1\beta}{\rho\mu}}\rho^2(\lambda+\delta)^2\xi_k^{1-2a}}+O(\frac{1}{(\lambda+\delta)^4\xi_k^{3-4a}}),
\end{eqnarray}
and thus
\begin{eqnarray}\label{g2-1}
|g_{2\pm}(\lambda)|=\frac{1}{4\sqrt{(\frac{\beta}{\mu}+\frac{\alpha}{\rho})^2-\frac{4\alpha_1\beta}{\rho\mu}}\rho^2m_j\xi_k^{2(1-a)}}+O(\frac{1}{\xi_k^{5-4a}}),~\lambda\in \Upsilon_k.
\end{eqnarray}
From \eqref{g1-1} and \eqref{g2-1}, we assert that \eqref{lam-12-abs} holds.
Thus, $g_1(\lambda)-g_2(\lambda)$ has only one zero in the interior of each $\Upsilon_k$. Therefore, we derive that   the spectrum $\{\lambda_{k,j,\pm}\}$ of $\mathcal{A}$ has the same asymptotic expression  as \eqref{lamada-3-23}. Hence, (\ref{alpha-lamada-4-23}) holds.

{ On the other hand, based on the expression \eqref{alpha-lamada-2-0}, we set the asymptotic expression of the spectrum $\{\lambda_{k,0}\}$ of $\mathcal{A}$ as follows:
\begin{align}
\label{lmk-va-2}&{\lambda_{k,0}}=-\delta+\varepsilon_{k},
\end{align}
{where $\varepsilon_k=O(\frac{1}{\xi_k^{1-a}})$}. Equivalent deformation of \eqref{alpha-eigen-dd1} leads to the following one:
\begin{eqnarray}\label{alpha-eigen-d2-2}
&&\frac{1}{\delta+\lambda}\bigg[\lambda^5+\lambda^4\delta+(\frac{\beta}{\mu}+\frac{\alpha}{\rho})\xi_k\lambda^3+\big[(\frac{\beta}{\mu}+\frac{\alpha}{\rho})\delta \xi_k-\frac{1}{\rho}\xi_k^a\big]\lambda^2 +\frac{\alpha_1\beta}{\rho\mu}\xi_k^2\lambda\cr
&&+\big(\frac{\alpha_1\beta\delta}{\rho\mu}\xi_k^2-\frac{\beta}{\rho\mu}\xi_k^{a+1}\big)\bigg]=0.
\end{eqnarray}
Then, inserting \eqref{lmk-va-2} into the equation $(\ref{alpha-eigen-d2-2})$, {we obtain the following estimate by direct calculation,}
\begin{eqnarray}\label{alpha-eigen-ss-2}
\varepsilon_k= \frac{1}{\alpha_1}\frac{1}{\xi_k^{1-a}}+{O(\frac{1}{{\xi_k}^{2-a}})}.
\end{eqnarray}
Substituting \eqref{alpha-eigen-ss-2} into \eqref{lmk-va-2}, we obtain the asymptotic expression \eqref{alpha-lamada-4-0}.
The proof is completed.}}
\end{proof}

{\begin{remark}\label{r41}
Solving the third equation of \eqref{alpha-eigen} directly and applying the condition $\eta(0)=0$, it gives \eqref{alpha-eta-1}, that is, $
\eta(s)=(1-e^{-\lambda s})v$.  Since  $\eta(s)$ is no longer belongs to the history memory space $L_{g}^2(\mathbb{R}_+,D(A^{\frac{a}{2}}))$ for any $\mathfrak{R}e\lambda\leq -\frac{\delta}{2}$, we get that  $\sigma_p(\mathcal{A})$, the set consisting of the point spectra (eigenvalues) of $\mathcal{A}$, must satisfies
$$
\sigma_p(\mathcal{A})\subset \{\lambda\in \mathbb{C}\big|-\frac{\delta}{2}<\mathfrak{R}e\lambda<0\}.
$$
Note from  the expression (\ref{alpha-lamada-4-0}) of $\lambda_{k,0}$  in Proposition \ref{alpha-p-4-2}  that
 there are infinitely many points  $\lambda_{k,0}$ satisfying $\mathfrak{R}e\lambda_{k,0}\leq {-\frac{\delta}{2}}$ and tending to $-\delta$. Hence,  they  are not the eigenvalues of $\mathcal{A}$.
\end{remark}}

Based on the asymptotic expressions (\ref{alpha-lamada-4-23}) of the eigenvalues of the system operator  as given in Theorem \ref{alpha-p-4-2},  one easily see that the imaginary axis is the asymptote of the eigenvalues of the system operator $\mathcal{A}$, which implies that the system can not achieve  exponential stability. Moreover,
 we can verify the optimality of the polynomial decay rates estimated in the last section. In fact,  by (\ref{alpha-lamada-4-23}), we see that there exists a branch of eigenvalues satisfies
{\begin{equation}
|\mathfrak{R}e\lambda_{k,j,\pm}|\sim |\mathfrak{I}m \lambda_{k,j,\pm}|^{-2(1-a)},\quad \mbox{as}\;\; k\to \infty.
\end{equation}}
Thus, by the above relationship along with the same argument as given in
 \cite[Corollary
 4.7]{haoliu2}, we can obtain
$$
\varlimsup\limits_{\lambda\in {\mathbb R}, |\lambda|\to \infty}
|\lambda |^{-(2-2a)}\| (i\lambda I - {\cal A})^{-1}\|_{{\cal L}({\cal H})}\ge C>0.
$$
Therefore,  by Lemma \ref{alpha-l-4-2},  we obtain the sharpness of  the obtained decay rate as given in Theorem \ref{alpha-t-4-2},
when the exponential-type kernel function $e^{-\delta s}$ is involved. We have the following result.
\begin{theorem}
Assume that the kernel function in system (\ref{alpha-sys1}) is exponential-type. Then, the rate
$t^{-\frac{1}{2-2a}}$ (as given in (\ref{2.5}))  is the optimal polynomial decay rate of the solutions to the system.
	\end{theorem}

\section{Conclusion}
\noindent

In this paper,  we investigate the asymptotic behavior of an abstract system of strongly coupled hyperbolic equations that contains a fractional operator in its infinite memory term. We examine how the stability of the strongly coupled system is affected by the presence of just one memory term.  We prove the well-posedness of the system under suitable Hilbert settings using semigroup theories. Moreover, we demonstrate that the solutions of the system converge to $0$ polynomially with a rate of $t^{-\frac{1}{2-2a}}$, which is dependent on the fractional order $a\in [0,1)$ in the memory. As the value of $a$ increases, the decay rate becomes faster, and when $a\to 1^-$, the system becomes the one with viscoelastic damping ($a=1$), the solutions  decays  exponentially.

We conduct a detailed spectral analysis and derive the asymptotic expressions of the eigenvalues of the system operator. Based on these expressions, we verify the optimality of the obtained decay rate $t^{-\frac{1}{2-2a}}$. Additionally, we discover that all the eigenvalues are distributed in the vertical strip  parallel to the imaginary axis.


Some open questions related to this issue under consideration are listed as follows:

\begin{itemize}
\item {We would like to point out that  it remains unclear how the decay rate changes when the memory term is localized to some sub-domain $\Omega_1\subset \Omega$.  In such a scenario, the decay rate would not only rely on the order $a$ of the memory term but also be influenced by the geometric conditions, such as the ``geometric control condition" (as mentioned in \cite{lebeau}). Investigating these effects will be worthwhile for future research.}

\item {It would be worth considering more general  coupled hyperbolic equations with infinite memory. In this case, the coupling terms  $\gamma\beta A p(t)$ and $\gamma\beta A v(t)$ in system (\ref{alpha-sys1})  could be replaced by more general forms $\gamma\beta A^b p(t)$ and $\gamma\beta A^b v(t)$, where $b\in [0,1]$. It appears that the decay rate may  depend on not only the order $a$ of the memory term but also  the order  $b$ of the coupling terms. However, it remains unknown what the explicit relationship is between the decay rates of the system and these two orders: $a$ and $b$.}

\item As  mentioned in Remark \ref{r41} that  the point spectra (eigenvalues) of $\mathcal{A}$ must lie within the ``strip" $ -\frac{\delta}{2}<\mathfrak{R}e \lambda<0$, because if $\mathfrak{R}e\lambda\leq -\frac{\delta}{2}$, the extended state $\eta$ is no longer in the state space for any value of $\lambda$.  This also implies that $\{ \lambda\in \mathbb{C}| \mathfrak{R}e\lambda\leq -\frac{\delta}{2}\}\subset \sigma(\mathcal{A})$.
 One see that there exist infinitely many spectra $\lambda_{k,0}$ as given in (\ref{alpha-lamada-4-0})  satisfying $ \mathfrak{R}e\lambda_{k,0}\leq -\frac{\delta}{2}$. So, they are not eigenvalues.  However,  careful analysis is required to further determine whether these spectra belong to the residual spectrum or the continuous spectrum.

	\end{itemize}



\end{document}